\documentclass[12pt]{amsart}
\usepackage{amssymb,amsthm}

\newtheorem{thm}{Theorem}
\newtheorem{theorem}{Theorem}[section]
\newtheorem{corollary}[theorem]{Corollary}

\newtheorem{lemma}[theorem]{Lemma}

\def\irr#1{{\rm  Irr}(#1)}

\title[Camina pairs]{On Camina pairs and the Fitting subgroup}

\begin{document}

\author[M. L. Lewis]{Mark L. Lewis}
\address{Department of Mathematical Sciences, Kent State University, Kent, OH 44242, USA}
\email{lewis@math.kent.edu}

\author[E. Pacifici]{Emanuele Pacifici}
\address{Dipartimento di Matematica e Informatica ``U. Dini" (DiMaI), Universit\`a degli Studi di Firenze, Viale Morgagni 67/A, 50134 Firenze, Italy}
\email{emanuele.pacifici@unifi.it}

\begin{abstract}  
In this paper, we consider the relationship between the Fitting subgroup and the Camina kernel of a Camina pair.  	
\end{abstract}

\subjclass[2010]{Primary 20C15; Secondary 20D25}

\keywords{Camina pairs, Fitting subgroup, Fitting height}

\thanks{The second author is partially supported by INdAM-GNSAGA. This research is also partially funded by the European Union-Next Generation EU, Missione 4 Componente 1, CUP B53D23009410006, PRIN 2022 2022PSTWLB - Group Theory and Applications.}

%\date{}

\maketitle

\begin{center}
This paper is dedicated to the memory of Francesco De Giovanni.
\end{center}

\section{Introduction}

In this paper, all groups are finite groups.  In \cite{camina}, A. Camina investigated the following condition that generalizes both Frobenius groups and extraspecial groups.  Suppose $N$ is a nontrivial normal subgroup of the group $G$ and the elements in $G \setminus N$ satisfy the condition that $x \in G \setminus N$ is conjugate to every element in the coset $xN$.  These pairs have been studied enough that they have been given the name Camina pairs.  

We say $(G,N)$ is a {\it Camina pair} if $N$ is a proper nontrivial normal subgroup in $G$ so that every element $x \in G \setminus N$ satisfies the condition that $x$ is conjugate to every element in the coset $xN$.  In this situation, we will say that $N$ is a {\it Camina kernel}.  We have written an expository paper \cite{my expos} which collects many of the results regarding Camina pairs.  We will review several key observations in Section \ref{characterize}.  

%In reviewing the literature, we found that it appears that people believe that $N$ must be solvable; however we do not find a proof in the literature, so we are including one here.

%\begin{thm} \label{solvable}
%If $(G,N)$ is a Camina pair, then $N$ must be solvable.
%\end{thm}

As we stated above, one motivation for Camina pairs is Frobenius groups, and it turns out that if $G$ is a Frobenius group with Frobenius kernel $N$, then $(G,N)$ is a Camina pair.   As it is known that $G/N$ need not be solvable for Frobenius groups, it follows that $G/N$ is not necessarily solvable for Camina pairs.  In fact, there exist Camina pairs that are not Frobenius groups where $G/N$ is nonsolvable (see \cite{nonsolv} and Theorem 1.20 in \cite{GGLMNT} for nonsolvable examples).

%On the other hand, we know when $G$ is a Frobenius group that its Frobenius kernel is nilpotent.  This raises the question of when $N$ is nilpotent for a Camina pair $(G,N)$.  We will show that either $N \le F(G)$ or $F (G) \le N$.  In fact, if $M$ is any normal subgroup of $G$, we have that $N \le M$ or $M \le N$. 

One motivation for writing this paper is our recent work with a number of coauthors where we consider the following generalization of Camina pairs. Given a nontrivial normal subgroup $N$ of the group $G$, we say that $(G,N)$ is an \emph{equal order pair}  if, for every $x\in G\setminus N$, the elements in the coset $xN$ have all the same order (see \cite{eop}). % of pairs $(G,N)$ that satisfy the condition that $N$ is a nontrivial normal subgroup of $N$ and every nontrivial coset of $N$ has the property that all of the elements have the same order.  When $(G,N)$ satisfies this condition, we say that $(G,N)$ is an equal order pair (see \cite{eop}).  
In Theorem~A of \cite{eop}, we show that if $(G,N)$ is an equal order pair, then either the Fitting subgroup $F(G) \le N$ or $N \le F (G)$.  This leads to partitioning equal order pairs into a trichotomy: (1) $F (G) < N$, (2) $F (G) = N$, and (3) $N < F(G)$.   

Now, for Camina pairs, there is a trichotomy that was proved by Camina in the original paper (see Theorem 2 of \cite{camina}). If $(G,N)$ is a Camina pair, then one of the following occurs: (1) $G$ is a Frobenius group with Frobenius kernel $N$, (2) $N$ is a $p$-group, or (3) $G/N$ is a $p$-group for some prime $p$.  We first note that unlike the trichotomy for equal order pairs, this trichotomy is not disjoint, although one can be more precise and make these categories disjoint.  Second, we will show that the fact that $F(G) \le N$ or $N \le F(G)$ is a consequence of a more general result about normal subgroups and Camina pairs.   

The point of this paper is to see how the previous trichotomy for Camina pairs lines up with the trichotomy for equal order pairs.  In particular, we prove the following:

%The point of this paper, is to show that the previous trichotomy for Camina pairs lines up with the trichotomy for equal order pairs and that the only time the Camina kernel is the Fitting subgroup is when $G$ is a Frobenius group whose Frobenius kernel is the Camina kernel.  %In our main theorem, we characterize when the Fitting subgroup of a Camina pair is nilpotent. 

\begin{thm} \label{intro-Fitting subgroup}
Let $(G,N)$ be a Camina pair.  Then the following hold:
\begin{enumerate}
\item $N = F(G)$ if and only if $G$ is a Frobenius group with Frobenius kernel $N$.
\item $N < F(G)$ if and only if $N$ is a $p$-group for some prime $p$ and $p$ divides $|G:N|$.  
\item $F(G) < N$ if and only if $G/N$ is a $p$-group for some prime $p$, $p$ divides $|N|$ and $G$ is not a $p$-group.    	
\end{enumerate}	
\end{thm}

We would like to thank Alan Camina for some helpful comments while preparing this paper.

%This next material ended up in the equal orders pair paper (so we omit it here):

%When $G$ is a Frobenius group with Frobenius kernel $N$ and $G/N$ is solvable, it is not difficult to see that $G/N$ has Fitting height at most $3$.  This opens the question of can we bound the Fitting height of $G/N$ when $(G,N)$ is a Camina pair and $G/N$ is solvable.  We show that we can find such a bound.

%\begin{thm}\label{Fitting height of G/N}
%Let $(G,N)$ be a Camina pair with $G/N$ solvable, then $G/N$ has Fitting height at most $4$.
%\end{thm} 

\section{Characterizing Camina pairs}\label{characterize}

%In this note, we investigate the relationship between Camina pairs and nilpotence of the Camina kernel.

%Let $G$ be a group and let $N$ be a normal subgroup of $G$.  We say $(G,N)$ is a {\it Camina pair} if for every element $x \in G \setminus N$, we have that $x$ is conjugate to all of $xN$.  (

In this section, we give more background on Camina pairs.  We also give and prove more complete descriptions of some families of Camina pairs than have appeared in the literature.   When $N$ is a normal subgroup of $G$, we write $\irr {G \mid N}$ for the set of irreducible characters of $G$ whose kernels do not contain $N$.  There are several equivalent conditions for the definition of Camina pair (see Lemma 4.1 of \cite{my expos}).  In this paper, we make use of the following equivalent conditions:

\begin{lemma} \label{cam equiv}
Let $N$ be a nontrivial normal subgroup of a group $G$.  Then the following are equivalent:
\begin{enumerate}
\item $(G,N)$ is a Camina pair.
\item If $x \in G \setminus N$, then $|C_G (x)| = |C_{G/N} (xN)|$.
\item For all $\chi \in \irr {G \mid N}$, then $\chi$ vanishes on $G \setminus N$. 
\end{enumerate}
\end{lemma}

Following \cite{waists}, when $G$ is a group, a {\it waist} for $G$ is a (normal) subgroup $M$ so that every normal subgroup $K$ satisfies that either $K \le M$ or $M \le K$.   It is proved in \cite{waists} that if $(G,N)$ is a Camina pair, then $N$ is a waist for $G$.

%\begin{lemma}
%Suppose $HK$ is a Frobenius group with Frobenius kernel $K$ and Frobenius complement $H$, 
%\end{lemma}

As we mentioned in the Introduction, it is known that if $(G,N)$ is a Camina pair, then one of the following occurs (see Theorem 2 of \cite{camina}): (1) $G$ is a Frobenius group with Frobenius kernel $N$, (2) $N$ is a $p$-group for some prime $p$, or (3) $G/N$ is a $p$-group for some prime $p$.  Since these three categories are overlapping, it is useful to refine these categories as follows:

When $(G,N)$ is a Camina pair, it satisfies exactly one of the following:
\begin{enumerate}
\item $G$ is a Frobenius group with Frobenius kernel $N$.
\item $G$ is a $p$-group for some prime $p$.
\item $N$ is a $p$-group and $p$ divides $|G:N|$, but $G/N$ is not a $p$-group for some prime $p$.
\item $G/N$ is a $p$-group and $p$ divides $|N|$, but $N$ is not a $p$-group for some prime $p$.
\end{enumerate}

In their paper \cite{CMS}, Chillag, Mann and Scoppola identify five types of Camina pairs. Now, they do not claim that they have identified all the possible types of Camina pairs, but we can use the ideas implicit in their five types to further refine our Categories (3) and (4).  In Category (3), it is often useful to distinguish between those in which $G$ is $p$-closed and those in which $G$ is not $p$-closed.   We note that Chillag, Mann and Scoppola state in \cite{CMS} that $(G,N)$ is a Camina pair of Type (3) with $G$ $p$-closed if and only if $G$ has a normal Sylow $p$-subgroup $P$ so that $(P,N)$ is a Camina pair and $G$ has a Hall $p$-complement $H$ so that $NH$ is a Frobenius group.  Since we have not found a proof of this result in the literature, we include one here.

%In \cite{kuisch}, it is conjectures if $(G,N)$ is a Camina pair so that $N$ is a $p$-group for some prime $p$ so that $p$ divides $|G:N|$ and $G$ is not $p$-group where $P$ is a Sylow $p$-subgroup of $G$ that $G$ is $p$-closed if and only if $(P,N)$ is a Camina pair, and proves this conjecture in a number a special cases.  We now prove a variation on this conjecture.

%\begin{lemma}\label{type 3 p closed}
%Let $(G,N)$ be a Camina pair of type 3.  Suppose $P$ is a Sylow $p$-subgroup of $G$ and $H$ is a Hall $p$-complement.  Then $(G,N)$ is a Camina pair with $G$ $p$-closed if and only if $P$ is normal in $G$, $(P,N)$ is a Camina pair and $NH$ is a Frobenius group. 
%\end{lemma} 

\begin{lemma}\label{type 3 p closed}
Let $G$ be a group, let $p$ be a prime, let $N$ be a nontrivial normal $p$-subgroup of $G$, and let $P$ be a Sylow $p$-subgroup of $G$.  Then $(G,N)$ is a Camina pair of Type (3) with $G$ $p$-closed if and only if $P$ is proper and normal in $G$, $(P,N)$ is a Camina pair, and $G$ has a Hall $p$-complement $H$ so that $NH$ is a Frobenius group.  	
\end{lemma}

\begin{proof}
%By \cite{kuisch}, we know if $(G,N)$ is a Camina pair, then $P$ is normal in $G$ if and only if $(P,N)$ is a Camina pair.  {}From \cite{CM}, we know that if $(G,N)$ is a Camina pair, then $NH$ is a Frobenius group.  
Suppose that $(G,N)$ is a Camina pair of Type (3) that is $p$-closed.  Then $P$ is clearly proper and normal in $G$ and properly contains $N$.  Moreover, we know from \cite{CM} that $NH$ is a Frobenius group.  
%Since $G$ is $p$-closed, we know that $P$ is normal in $G$.  
Consider a character $\phi \in \irr {P \mid N}$.  Fix $\nu \in \irr N$ to be an irreducible constituent of $\phi_N$ and take $\chi \in \irr G$ to be an irreducible constituent of $\phi^G$.  Write $T$ for the stabilizer of $\nu$ in $G$, and since $NH$ is a Frobenius group, we see that $T \le P$.  Because $(G,N)$ is a Camina pair, we know that $\chi$ vanishes on $G \setminus N$.  This implies that $\chi$ is the unique irreducible constituent of $\nu^G$.  Let $\hat\chi \in \irr {T \mid \nu}$ be the Clifford correspondent for $\chi$ with respect to $\nu$.  In light of the Clifford correspondence, $\hat\chi$ will be the unique irreducible constituent of $\nu^T$.  Using the Clifford correspondence in $P$, we see that $(\hat\chi)^P$ is the unique irreducible constituent of $\nu^P$.   By the uniqueness, we have $\phi = (\hat\chi)^P$, and hence, $\phi$ vanishes on $P \setminus N$.  Since $\phi$ was arbitrary in $\irr {P \mid N}$, we conclude that $(P,N)$ is a Camina pair.

Conversely, assume that $P$ is normal and proper in $G$, $(P,N)$ is a Camina pair, and $H$ is a Hall $p$-complemenet of $G$ so that $HN$ is a Frobenius group.  Consider a character $\chi \in \irr {G \mid N}$, and take $\theta$ to be an irreducible constituent of $\chi_P$ and $\nu$ be an irreducible constituent of $\theta_N$.  Observe that $\nu \ne 1_N$.  Since $(P,N)$ is a Camina pair, we know that $\theta$ vanishes on $P \setminus N$, and this implies that $\theta$ is the unique irreducible constituent of $\nu^P$.   Because $NH$ is a Frobenius group, we have that $N$ is the stabilizer of $\nu$ in $NH$.  The uniqueness of $\theta$ as a constituent of $\nu^P$ will imply that $P$ is now the stabilizer of $\theta$ in $PH = G$.  It follows that $\theta^G$ is irreducible, and since $\chi$ is a constituent of $\theta^P$, we obtain $\chi = \theta^G$. %Using this uniqueness and the Glauberman-Isaacs correspondence, since $N$ is the stabilizer of $\nu$ in $NH$, it follows that $P$ is the stabilizer of $\theta$ in $G$, and so, $\chi = \theta^G$.  
Let $x \in G \setminus N$.  If $x \in G \setminus P$, then since $\chi$ is induced from $P$, we have $\chi (x) = 0$ and if $x \in P \setminus N$, then since $\chi$ is induced from $\theta$ and $\theta$ vanishes on $P \setminus N$, we see that $\chi (x) = 0$.  This implies that $\chi$ vanishes on $G \setminus N$.  Because $\chi$ was arbitrary, we conclude that $(G,N)$ is a Camina pair of Type (3) with $G$ $p$-closed.
\end{proof}

Obviously, if $G$ is $p$-closed then $G$ is $p$-solvable, so the Camina pairs of Type (3) that are $p$-closed are necessarily $p$-solvable.  On the other hand, the groups in Theorem 1.20 of \cite{GGLMNT} give examples of Camina pairs $(G,M)$ of Type (3) where $M$ is a $p$-group, $G$ is $p$-solvable, but not solvable.  Additionally, in \cite{nonsolv} is  presented an example of a Camina pair $(G,M)$ of Type (3) where $M$ is a $p$-group and $G$ is not $p$-solvable.  In particular, $G$ is not $p$-closed.  We would not be surprised that when $(G,M)$ is a Camina pair of Type (3) where $G$ is not solvable and not  $p$-closed, then $G$ must not be $p$-solvable, but we have not proved this at this time.

Suppose $(G,N)$ is a Camina pair of Type (3) where $G$ is not $p$-closed.  We know from \cite{CM} that if $H$ is a $p'$-subgroup of $G$, and in particular, if $G$ has a Hall $p$-complement $H$, then $HN$ is a Frobenius group.  There is a conjecture in \cite{kuisch} that suggests that $(P,N)$ will not be a Camina pair where $P$ is a Sylow $p$-subgroup of $G$, however we believe this is still an open question.  %We do know from \cite{CM} that $N < O_p (G) = F(G)$.   
Another question that arises is whether one can obtain a characterization of these groups.  At this time, we do not have one.

We now consider the case where $(G,N)$ is a Camina of Type (4). %pair such that $G/N$ is a $p$-group, $N$ is not the Frobenius kernel of $G$, and $G$ not a $p$-group.  
In \cite{CMS}, Chillag, Mann and Scoppola show that $(G,P,P \cap N)$ is a Frobenius-Wielandt triple.  We extend this to a necessary and sufficient condition.

\begin{theorem} \label{G/N p-group}
Let $p$ be a prime, let $G$ be a group with a normal subgroup $N$ and Sylow $p$-subgroup $P$.  Then $(G,N)$ is a Camina pair of Type (4) where $G/N$ is a $p$-group 
%where $G/N$ is a $p$-group with $N$ not the Frobenius kernel of $G$ and $G$ not a $p$-group 
if and only if $(G,P,P \cap N)$ is a Frobenius-Wielandt triple, $(P,P \cap N)$ is a Camina pair, and $G$ has a normal $p$-complement properly contained in $N$.
\end{theorem}	

\begin{proof}
Suppose $(G,N)$ is a Camina pair of Type (4).  By Lemma 2.1 (1) of \cite{CMS} or Theorem C of \cite{coprime}, $G$ has a normal $p$-complement $K$.  Note that if $K = N$, then $N$ would be a Frobenius kernel of $G$ (see Proposition 1 of \cite{camina}), and this would violate $(G,N)$ being of Type (4).  It follows that $K < N$.  Let $P$ be a Sylow $p$-subgroup of $G$, and observe that $P \cong G/K$ and $(G/K,N/K)$ is a Camina pair; so $(P,P \cap N)$ is a Camina pair.  Let $x$ be an element of $G \setminus N$.  We know that $x$ is conjugate to every element of $xN$.  Since $G = NP$, we know that $xN$ intersects $P$ nontrivially.  Thus, $x$ is conjugate to an element of $P \setminus P \cap N$.  Hence, every element of $G \setminus N$ is conjugate to an element of $P \setminus P \cap N$.  This implies that $(G,P,P\cap N)$ is a Frobenius-Wielandt triple by Lemma 3.3 of \cite{lewis}.

Conversely, suppose $(G,P,P \cap N)$ is a Frobenius-Wielandt triple, $(P,P \cap N)$ is a Camina pair, and $G$ has a normal $p$-complement $K$ so that $K < N$.  Observe that $P$ is a complement for $K$ in $G$.  Let $x$ be an element of $G \setminus N$.  We show $xn$ is conjugate to some element in $xN \cap P$ for every element $n \in N$.  Notice that $G$ and $K \le N$ satisfy the hypotheses of Theorem 3.5 of \cite{burkett} with $P$ a complement for $K$.  The fact that $(G,P,P \cap N)$ is a Frobenius-Wielandt triple is condition (4) of Theorem 3.5 of \cite{burkett}.  Let $y \in xnK \cap P$, and observe that $xn K = yK$.  Condition (6) of Theorem 3.5 of \cite{burkett} states that $y$ is conjugate to every element of $yK$.  Thus, $xn$ is conjugate to $y \in xnK \cap P \subseteq xnN \cap P = xN \cap P$.  Now, suppose $y \in xN \cap P$.  We next show that $xN \cap P = y(P \cap N)$.  We have $yN = xN$.  We see that $y(P \cap N) \subseteq yN \cap P$.  Suppose $g \in yN \cap P$.  We see that $g = yn$ for some $n \in N$.  Since $yn = g \in P$ and $y \in P$, we have $n \in P$, so $n \in P \cap N$, and thus, $g \in y(P \cap N)$ and so, $g(P \cap N) = y (P \cap N)$.  It follows that $xN \cap P = yN \cap P = y(P \cap N)$.  We have now shown that every element in $xN$ is conjugate to some element in $y (P \cap N)$.  Since $(P,P \cap N)$ is a Camina pair, all elements of $y(P \cap N)$ are conjugate (in $P$).  This implies that all elements of $xN$ are conjugate (in $G$).  Thus, $(G,N)$ is a Camina pair.
\end{proof}

\section{Camina pairs and the Fitting subgroup}

We start working toward a proof of Theorem \ref{intro-Fitting subgroup}.  Note that if $(G,N)$ is a Camina pair of Type (1) (i.e., $G$ is a Frobenius group with Frobenius kernel $N$), then we have $N = F(G)$.  Also, if $(G,N)$ is a Camina pair of Type (2) (i.e. $G$ is a $p$-group), then clearly $N < F(G) = G$.  Therefore, the theorem will follow by showing that if $(G,N)$ is a Camina pair of Type (3) then $N < F(G)$, whereas if $(G,N)$ is a Camina pair of Type (4) then $F (G) < N$.   This will be done in Theorem \ref{type 3 proper} and Corollary \ref{type 4 proper}, respectively.

%We now show that the Fitting subgroup of $G$ when $(G,N)$ is a Camina pair is determined by what Type of Camina pair we have. Suppose first that $G$ is a Frobenius group with Frobenius kernel $N$.  We know in this case that $N = F (G)$.  Suppose $G$ is a $p$-group for some prime $p$, then obviously $G = F(G)$ and $N < F(G)$, and so the result holds when $G$ is a $p$-group.  Thus, we may assume $G$ is not a $p$-group.  We first consider the case where $(G,N)$ is a Camina pair where $G$ is not a Frobenius group with Frobenius kernel $N$ and $G$ is not a $p$-group, but $N$ is a $p$-group for some prime $p$.  In other words, we consider Camina pairs of Type 3.  We will show that $F(G) < N$.  

To prove Theorem \ref{type 3 proper}, we will need some results about nonsolvable Camina pairs that have been proved in \cite{nonsolv} and we use a couple of results regarding equal order pairs in our recent paper \cite{eop}.  Recall that $(G,N)$ is an {\it equal order pair} if $N$ is a proper, nontrivial normal subgroup of $G$ and every coset $aN \ne N$ of $N$ in $G$ has the property that all of the elements of $aN$ have the same order.  It is not difficult to see that Camina pairs have this property, so equal order pairs are a generalization of Camina pairs.

In \cite[Remark~5.5]{eop}, we proved that if $N$ is the natural module for ${\rm SL} (2,2^n)$ and $G = N \rtimes {\rm SL} (2,2^n)$, then $(G,N)$ is not an equal order pair.  In Theorem \ref{type 3 proper}, we will make use of the following generalization of this result.  To prove this result, we need one more definition.  We say the prime $p$ is a {\it primitive prime} for $a^n - 1$ where $a$ and $n$ are positive integers if $p$ divides $a^n-1$ and does not divide $a^k -1$ for all $1 \le k < n$.  (In many places, these are called Zsigmondy prime divisors, see page 94 of \cite{MaWo}.)  In the Zsigmondy prime theorem, Zsigmondy proves that primitive primes exist for all pairs $(a,n)$ except for when $a$ is a Mersenne prime and $n = 2$ and $(a,n) = (2,6)$.  For a nice overview with proofs of the number theory involved with the Zsigmondy prime theorem, we suggest the reader consult Sections 6.1 and 6.2 of \cite{numth}.    %We also need to know the similar result for the nonsplit extension which we prove next.
%At the moment this is not proved.  We need to prove that the nonsplit extension of the natural module by ${\rm SL} (2,2^n)$ is not an equal order pair.

\begin{lemma}\label{eop ext}
Let $G$ be a group, and $N$ an elementary abelian normal $2$-subgroup of $G$ such that $G/N\cong{\rm SL} (2,2^n)$ (for $n\geq 2$) and $C_G(N)=N$. Assume that $N$ is a completely reducible $G/N$-module over ${\rm{GF}}(2^n)$ whose irreducible constituents are all isomorphic to the natural module for ${\rm SL} (2,2^n)$. Then $(G,N)$ is not an equal order pair.
\end{lemma}

\begin{proof} For a proof by contradiction, assume that $(G,N)$ is an equal order pair. We consider first the case $n \neq 3$, so there exists a primitive prime divisor $p$ of $2^{2n} - 1$. Since $G/N$ has elements of order $p$, we can take an element $y \in G$ having order $p$ as well, and we set $Y = \langle y \rangle$. Note that any irreducible constituent $M$ of the $G/N$-module $N$ is a faithful $G/N$-module, hence $Y$ acts faithfully on it.  Note that since $M$ is isomorphic to the natural module, we have $|M| = p^{2a}$.  Moreover, the fact that $p$ is a primitive prime divisor of $|M| - 1$ implies that $M$ is in fact irreducible also as a $Y$-module. Note that $N_N (Y)$ and $Y$ are both normal subgroups of $N_G(Y)$ having coprime orders, hence they centralize each other; in particular, $N_N (Y)$ is a $Y$-submodule of $N$ on which $Y$ acts trivially. But since the $Y$-module $N$ does not have any irreducible constituent that is the trivial module for $Y$, we must have $N_N(Y)=1$. 

Next, observe that $$N_G(Y)=N_G(Y)/N_N(Y)\cong N_G(Y)N/N=N_{G/N}(YN/N)$$ is a dihedral group, thus it contains an element $x$ of order $2$. Since $x$ lies in $G\setminus N$, our assumption that $(G,N)$ is an equal order pair yields that every element in the coset $xN$ has order $2$. But now every element of $\langle x\rangle N$ has order $2$, and therefore this subgroup is abelian. As a consequence, we get the contradiction that $x$ lies in $C_G(N)\setminus N$, and we are done.

It remains to treat the case $n=3$.  In this situation, we will show that we can apply the same argument as above by choosing $y \in G$ as an element such that $yN$ generates a Sylow $3$-subgroup of $G/N$ (which is cyclic of order $9$), and using the fact that $6$ is the smallest positive integer $a$ such that $9$ divides $2^a-1$.  To do this, we need to observe that as above any irreducible constituent $M$ of the $G/N$ module $N$ is a faithful $G/N$-module.  In particular, $|M| = 2^6$. Thus, if we take $Z = \langle y^3 \rangle$, then $Z$ does not centralize all of $M$.   Furthermore, $|[M,Z]| - 1$ is divisible by $9$, and since $|M|$ is the smallest power of $2$ so that $9$ divides $2^a - 1$, we have $M = [M,Z]$.  It follows that $Y$ is a Frobenius complement.  Now, we can use a similar proof as for a primitive prime. The proof is complete. 
\end{proof}

%\begin{lemma} \label{eop ext}
%Let $G$ be a group so that has a normal subgroup $N$ so that $G/N \cong {\rm SL} (2,2^n)$, $N$ is isomorphic to the natural module for $G/N$, and the extension is not split.  Then $(G,N)$ is not an equal order pair.
%\end{lemma}

%\begin{proof}
%Suppose $(G,N)$ is an equal order pair.  Let $p$ be a Zsigmondy prime divisor of $2^{2n} - 1$.  Let $P$ be a Sylow $p$-subgroup of $G$.  Since $2^{2n}-1$ is the smallest integer of the form $2^a - 1$ divisible by $p$ with $a \ge 2$, we see that $N_G (P) \cap N = 1$.  We know that $N_{G/N} (PN/N)$ is a dihedral group of order $2|P|$.  Thus, we see that $N_G (P) \cong N_{G/N} (PN/N)$, and hence, $N_G (P)$ is a dihedral group.  In particular, $G \setminus N$ has an element $x$ of order $2$.  Since $(G,N)$ is an equal order pair, every element in the coset $Nx$ will have order $2$.  Since $N$ is elementary abelian and $|N\langle x \rangle:N| = 2$, we conclude that every element in $N \langle x \rangle$ has order $2$, and so, $N \langle x \rangle$ is abelian.  We now have $x \in C_G (N)$.  Now, $N < C_G (N)$ is normal in $G$.  Since $N$ is not central in $G$, we have that $C_G (N) < G$.  In particular, $C_G(N)/N$ is a nontrivial proper normal subgroup of $G/N$ which is a contradiction since $G/N$ is simple.  
%\end{proof}

We are now ready to prove Theorem \ref{type 3 proper}.

%The following is inspired by Theorem 5.5 of \cite{gagola}.   This is not currently proved.

%\begin{lemma}
%If $(G,N)$ is a Camina pair such that $G$ is not a Frobenius group with Frobenius kernel $N$ and $G$ is not a $p$-group but $N$ is a $p$-group, then $N < F(G) = O_p (G)$.
%\end{lemma}

%Here is what I can currently prove:

\begin{theorem} \label{type 3 proper}
If $(G,N)$ is a Camina pair of Type (3), %such that $G$ is not a Frobenius group with Frobenius kernel $N$ and $G$ is not a $p$-group but $N$ is a $p$-group, 
then $N < F(G) = O_p (G)$.
%one of the following occurs:
%\begin{enumerate}
%\item $N < F(G) = O_p (G)$.
%\item $p = 2$, $N = F(G) = O_2 (G)$, any minimal normal subgroup of $G/N$ is of the form either ${\rm SL}_2 (2^n)$ for some integer $n \ge 2$ or ${\rm Sz} (2^{2m+1})$ for some integer $m \ge 1$.  Furthermore, $G/N$ has at most one minimal subgroup of each form %and when both forms occur, it must be that ${\rm GCD} (n,2m+1) = 1$.  Finally, if $S/N = {\rm Soc} (G/N)$, then $S/N$ is one of (i) ${\rm SL}_2 (2^n)$, (ii) ${\rm Sz} (2^{2m+1})$, or (iii) ${\rm SL}_2 (2^n) \times {\rm Sz} (2^{2m+1})$ and $G/S$ is solvable.
%\end{enumerate}     
\end{theorem}

\begin{proof}
%We suppose that $N$ is a $p$-group for some prime $p$ and that $p$ divides $|G:N|$ in this case.  
%Suppose $(G,N)$ is a Camina pair of Type (3).  
We work by induction on $|G|$.   Suppose there exists $1 < M < N$ so that $M$ is normal in $G$.  We know that $(G/M,N/M)$ is a Camina pair of Type (3).  Let $F/M = F(G/M)$.  By the inductive hypothesis, we have $N/M < F/M$ and $F/M = O_p (G/M)$.  Since $M$ is a $p$-group, it is not difficult to see that $F = F(G) = O_p (G)$, and we have $N < F (G)$ as desired.  Thus, we may assume that $N$ is a minimal normal subgroup of $G$.

By Lemma 4.4 of \cite{CM}, we have $O_{p'} (G/N) = 1$.  %If $G/N$ is $p$-solvable, then 
If $O_p (G/N) > 1$, then $F(G/N) = O_p (G/N) = O_p (G)/N$, and so, $N < O_p (G)$ and $F(G) = O_p (G)$.  Thus, we may assume $O_p(G/N) = 1$, so that $G/N$ is not $p$-solvable.  
%We also assume that $O_p (G/N) = 1$ and $O_{p'} (G/N) = 1$ and 
We work to find a contradiction.  This yields $N = F(G) = O_p (G)$, and if $M/N$ is a minimal normal subgroup of $G/N$, then $M/N$ is nonabelian and has order divisible by $p$.  In particular, $M/N$ is a direct product of some nonabelian simple group.  This implies that $2$ divides $|G:N|$.

Suppose $p$ is odd.  Since $2$ divides $|G:N|$, we see that $G$ contains an involution $x$.  Since $N \langle x \rangle$ must be a Frobenius group, we see that $x$ inverts every element of $N$.  This implies $x$ corresponds to a central automorphism of $N$.  By Proposition 2.2 of \cite{nonsolv}, we know that $N = C_G (N)$.   It follows that $N \langle x \rangle$ is normal in $G$.  By the Frattini argument, we have $G = N N_G (\langle x \rangle)$ and it easy to see that $N \cap N_G (\langle x \rangle) = 1$.  This implies that $N$ is complemented in $G$.   We know (Theorem 4.5 of \cite{CM}) that a complemented Camina pair must be a Frobenius group, and we have a contradiction.  

Thus, we must have $p = 2$.  %Note that if $O_2 (G/N) > 1$, then (1) occurs, so we may assume that $O_2 (G/N) = 1$.  Since we have seen that $O_{2'} (G/N) = 1$, we know $G/N$ is nonsolvable.  
We now have $G/N$ is nonsolvable with $O_{2'} (G/N) = 1$ and $O_2 (G/N) =1$.  Thus, we may use Theorem 1.3 of \cite{nonsolv} to see that the solvable residual of $G/N$ (which we denote by $(G/N)^\infty$) must be ${\rm SL}(2, 2^n)$ for $n \ge 2$.  Furthermore, by Corollary 1.4 of \cite{nonsolv}, we have that $N$ will be a direct sum of natural modules for $(G/N)^\infty$. Now, by Lemma~\ref{eop ext}, we have that the extension of $N$ by $(G/N)^\infty$ does not yield an equal order pair together with $N$.

On the other hand, we note following the definition of equal order pairs in \cite{eop} that Camina pairs are equal order pairs and using Lemma 2.1 of that paper, we see that if $X/N = (G/N)^\infty$, then $(X,N)$ is an equal order pair.  But this contradicts the  above observation, and  
%we see that $O_2 (G/N) = 1$ cannot occur.  Thus, we must have $O_2 (G/N) > 1$, and 
the result is proved.
\end{proof}

We now refine the structure of these groups.  In particular, we obtain more information about Camina pairs of Type (4).  We note that by Lemma 2.1 (3) of \cite{CMS} we know that $N$ in this situation is solvable. 

\begin{lemma} \label{O_p (G)}
Let $(G,N)$ be a Camina pair of Type (4).  Suppose $p$ is the prime so that $G/N$ is a $p$-group and take $P$ to be a Sylow $p$-subgroup of $G$.
%where $G/N$ is a $p$-group  and $G$ not a $p$-group.  
Then $O_p (G) = C_P (O_{p'} (G)) = \cap_{v \in O_{p'} (G)} C_P (v) < P \cap N$.
\end{lemma}

\begin{proof}
We know by Theorem C of \cite{coprime} or Lemma 2.1 (1) of \cite{CMS} that since $(G,N)$ is a Camina pair of Type (4), $G$ has a normal $p$-complement.  Thus, $O_{p'} (G)$ is the normal $p$-complement of $G$, and so, it is the normal $p$-complement of $N$.  In particular, $G = O_{p'} (G) P$.  We know if $x \in G \setminus N$, then $|C_G (x)| = |C_{G/N} (xN)|$, so $C_G (x)$ is a $p$-group for all $x \in G \setminus N$.  Hence, if $v \in O_{p'} (N) \setminus \{ 1 \}$, then $C_G (v) \le N$, and hence, $C_P (v) \le P \cap N$.  We see that $\cap_{v \in O_{p'} (G)} C_P (v) \le P \cap N$.  We have that $C_G (O_{p'} (G)) = \cap_{v \in O_{p'} (G)} C_G (v)$ and $C_P (O_{p'} (G)) = P \cap C_G (O_{p'} (G)) = P \cap (\cap_{v \in O_{p'} (G)} C_G (v)) = \cap_{v \in O_{p'} (G)} P \cap C_G (v) = \cap_{v \in O_{p'} (G)} C_P (v)$.  

%Since $O_p (G)$ and $O_{p'} (G)$ are both normal and coprime orders, we have 
Clearly, $O_p (G) \le C_P (v)$ for all $v \in O_{p'} (G)$, so $O_p (G) \le \cap_{v \in O_{p'} (G)} C_P (v)$.  On the other hand, $C_G (O_p' (G))$ is normal in $G$, so $C_P (O_{p'} (G))$ is normal in $P$.  Also, $O_{p'} (G)$ centralizes and so normalizes $C_P (O_{p'} (G))$.  Since $G = O_{p'} (G) P$, we conclude that $C_P (O_{p'} (G))$ is normal in $G$, and so, $C_P (O_{p'} (G)) \le O_p (G)$.  We conclude that $C_P (O_{p'} (G)) = O_p (G)$.  

It remains to show that $O_p (G) \ne P \cap N$.  To do this, we may suppose $P \cap N = O_p (G)$ and $P \cap N$ is a minimal normal subgroup of $G$.  This implies that $(G/(P \cap N),N/(P \cap N))$ is Camina pair.  Since $(|N:(P \cap N)|,|G:N|) = 1$, this implies that $G/(P \cap N)$ is a Frobenius group.  Hence, $P/(P \cap N)$ is either a cyclic group or a generalized quaternion group.  
%Without loss of generality, we may assume that $P \cap N$ is minimal normal in $G$.  This implies that $P \cap N \le Z(P)$.  

Recall that $G = O_{p'} (G) P$.  Since $O_p (G) > 1$ and $O_p (G)$ is normal in $P$, we have $O_p (G) \cap Z(P) > 1$.  Obviously, both $P$ and $O_{p'} (G)$ centralize $O_p (G) \cap Z(P)$.  Hence, $O_p (G) \cap Z(P)$ is central in $G$.  The minimality of $O_p (G) = P \cap N$ implies that $O_p (G) = O_p (G) \cap Z(P)$, and so, $O_p (G) \le Z(P)$.    

Let $1 \ne \lambda \in \irr {P \cap N}$.  Then $\lambda$ extends to $\theta \in \irr P$ since $P/(P \cap N)$ is either cyclic or generalized quaternion.  It follows that $1_{O_{p'} (N)} \times \lambda \in \irr N$ will extend to $\chi \in \irr G$ and this is a contradiction since $\chi \in \irr {G \mid N}$ will be linear, and thus, it will not vanish on $G \setminus N$.  This is a contradiction.  Thus, we must have $O_p (G) < P \cap N$.  
\end{proof}

We are now able to show that if $(G,N)$ is a Camina pair of Type~(4), then $F(G) < N$.

\begin{corollary} \label{type 4 proper}
If $(G,N)$ is a Camina pair of Type (4),
% such that $G$ is not a Frobenius group with Frobenius kernel $N$ and $G$ is not a $p$-group but $G/N$ is a $p$-group for some prime $p$, 
then $F(G) < N$.  
\end{corollary}

\begin{proof}
%If $G$ is a Frobenius group whose Frobenius complement is quaternion of order $8$ and $F = F(G)$ is the Frobenius kernel.   In this case, $N = G'$ has index $4$, and we know that $F = F(G) < N$ has index $2$ in $N$.  Hence, we may assume that $G$ is not a Frobenius group.  
By Lemma \ref{O_p (G)}, we know that $O_p (G) < P \cap N$ where $P$ is a Sylow $p$-subgroup.  If $N$ is nilpotent, we have $P \cap N = O_p (N) = O_p (G)$ which is a contradiction.  Therefore, as $N$ is a waist, we conclude that $F(G) < N$.	
\end{proof}

\end{document}